\documentclass[11pt]{amsart}
\usepackage[a4paper,margin=1in]{geometry}
\usepackage[T1]{fontenc}
\usepackage{lmodern}
\usepackage{microtype}
\usepackage{amsmath,amssymb,amsthm,mathtools,mathrsfs}
\usepackage{tikz}
\usepackage{graphicx}
\usepackage{float}
\usepackage{array}
\usetikzlibrary{positioning,calc,arrows.meta}
\usepackage[dvipsnames]{xcolor}
\usepackage{hyperref}
\hypersetup{
  hidelinks,
  pdfauthor={Qi Guo; Xueping Huang; Yi C. Huang; Fanghua Lin; Juncheng Wei},
  pdftitle={Br\'ezis-Coron Uniqueness of Harmonic Maps via Boundary Rigidity of Hopf Differentials},
  pdfsubject={Harmonic maps}
}
\numberwithin{equation}{section}
\newtheorem{theorem}{Theorem}[section]
\newtheorem{proposition}[theorem]{Proposition}
\newtheorem{lemma}[theorem]{Lemma}

\newtheorem{corollary}[theorem]{Corollary}
\theoremstyle{remark}
\newtheorem*{remark}{Remark}

\newcommand{\B}{B_1}
\newcommand{\R}{\mathbb{R}}
\newcommand{\C}{\mathbb{C}}
\newcommand{\Sph}{\mathbb{S}^2}

\DeclareMathOperator{\Area}{Area}
\newcommand{\AlgArea}{\operatorname{Area}_{\mathrm{alg}}}
\DeclareMathOperator{\Length}{Length}

\newcommand{\RR}{\mathbb R}
\newcommand{\eiii}{e_3}
\newcommand{\Rot}{\mathcal R}

\title[Br\'ezis-Coron Uniqueness of Harmonic Maps]{On Br\'ezis' Open Problem 3.1 }

\author[Q. Guo]{Qi Guo}
\address{School of Mathematics, Renmin University of China, Beijing, 100872, P.R. China}
\email{qguo@ruc.edu.cn}

\author[X. Huang]{Xueping Huang}
\address{School of Mathematics and Statistics, Nanjing University of Information Science and Technology, Nanjing 210044, P.R. China}
\email{hxp@nuist.edu.cn}

\author[Y. C. Huang]{Yi C. Huang}
\address{School of Mathematical Sciences, Nanjing Normal University, Nanjing 210023, P.R. China}
\email{Yi.Huang.Analysis@gmail.com}

\author[F. Lin]{Fanghua Lin}
\address{Courant Institute of Mathematical Sciences, New York University, New York, NY 10012, USA}
\email{linf@math.nyu.edu}

\author[J. Wei]{Juncheng Wei}
\address{Department of Mathematics, Chinese University of Hong Kong, Shatin, N.T., Hong Kong}
\email{wei@math.cuhk.edu.hk}

\thanks{The research of J. Wei is partially supported by GRF grant of HK RGC entitled "On critical and supercritical Fujita equations". X.
Huang was supported by the National Natural Science Foundation of China (Grant No. 12371206). We thank Dr. Yifu Zhou and Dr. Liming Sun for many
useful discussions.}

\begin{document}

\begin{abstract}
Let $B_1$ be the unit disk in ${\mathbb R}^2$. We consider the harmonic map equation
\[
 -\Delta u=|\nabla u|^2u,
 \]
subject to the Dirichlet boundary condition $
 u(e^{i\theta})=(R\cos\theta,R\sin\theta,\sqrt{1-R^2}):=g_R$, where $0<R<1$ and $u\colon B_1\to\mathbb S^2$ is understood in the weak
 harmonic-map sense. In 1983, Br\'ezis and Coron (\cite{BrCo}) proved the existence of two explicit solutions of this nonlinear Dirichlet
 problem and showed that they are the unique minimizers in their respective  relative homotopy classes.
 In this paper, we resolve a long-standing open question originally posed in their work, later posed as {\bf Open Problem 3.1} in Brézis’
 {\em Favorite Open Problems List} (\cite{Brezis2023}). Specifically, we prove that these two explicit maps are the {\bf only} weak harmonic maps with boundary
 trace $g_{R}$, thereby providing a definitive affirmative answer to Brézis' open problem.

The proof is based on a boundary rigidity argument. An auxiliary potential $X$ associated with $u$, the Pohozaev identity for the Hopf
differential, and the planar isoperimetric inequality imply
\[
|u_r|\equiv R,
\qquad
u_r\cdot u_\theta\equiv0
\qquad\text{on }\partial B_1.
\]
Thus the Hopf differential vanishes on the boundary and hence, by holomorphicity, on the whole disk. The problem is then reduced to the
conformal case, where a stereographic-coordinate classification gives exactly the two Br\'ezis--Coron maps.
\end{abstract}

\maketitle
\tableofcontents

\section{Introduction}

Let $B_1$ be the unit disk in $ {\mathbb R}^2$.
In this paper, we study weak harmonic maps $u\colon B_1 \to\Sph$ with trace $g_R$, equivalently weak solutions of
\[
 -\Delta u=|\nabla u|^2u
 \quad\text{in }B_1,
 \qquad
 u|_{\partial B_1}=g_R,
\]
where $g_R:=g_R(e^{i\theta})= \bigl(R\cos\theta,R\sin\theta,\sqrt{1-R^2}\bigr),$ $\theta\in[0,2\pi]$, $R\in (0,1)$.
The boundary curve is a latitude circle traversed once with constant speed. Br\'ezis and Coron exhibited two explicit solutions and proved
that they are the energy minimizers in the two relevant relative homotopy classes of degree $\pm 1$ \cite{BrCo}. A long-standing open question
is whether the same boundary trace admits any further, necessarily non-minimizing, critical point; Br\'ezis later posed this as {\it Open
Problem~3.1} in his {\em "Favorite Open Problems"} list \cite{Brezis2023}.

problem in which one has to decide whether the two natural conformal and anti-conformal fillings exhaust all critical points.

To simplify the notations, we set
\[
 \rho=\frac{R}{1+\sqrt{1-R^2}}\in(0,1).
\]
The two Br\'ezis--Coron solutions are
\[
 u^+(r,\theta)=
 \left(
 \frac{2\rho r}{1+\rho^2r^2}\cos\theta,
 \frac{2\rho r}{1+\rho^2r^2}\sin\theta,
 \frac{1-\rho^2r^2}{1+\rho^2r^2}
 \right),
\]
and
\[
 u^-(r,\theta)=
 \left(
 \frac{2\rho r}{\rho^2+r^2}\cos\theta,
 \frac{2\rho r}{\rho^2+r^2}\sin\theta,
 \frac{r^2-\rho^2}{\rho^2+r^2}
 \right).
\]
Let $S=(0,0,-1)$ be the south pole. We use the south-pole stereographic projection
\[
 \pi_S(a,b,c)=\frac{a+ib}{1+c},
 \qquad
 \pi_S^{-1}(w)=\left(
 \frac{2\Re w}{1+|w|^2},
 \frac{2\Im w}{1+|w|^2},
 \frac{1-|w|^2}{1+|w|^2}
 \right).
\]
In this coordinate, the two solutions correspond to
\[
 w^+(z)=\rho z,
 \qquad
 w^-(z)=\frac{\rho}{\bar z}.
\]
A direct verification of these formulas and of the associated potentials is recorded in Appendix~\ref{app:explicit}. Figure~\ref{fig:problem}
sketches the geometry of the boundary trace and the two explicit fillings.

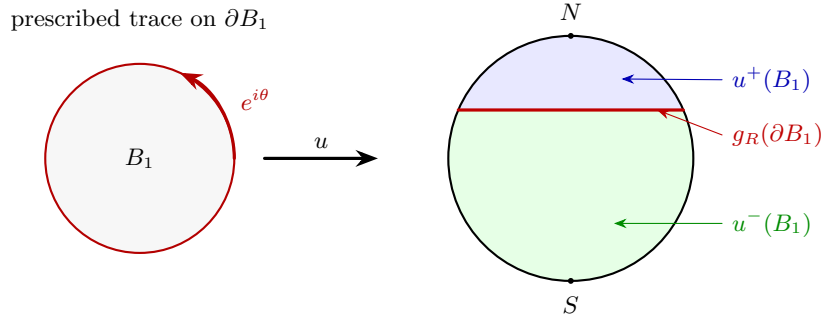
\begin{figure}[H]
\centering
\begin{tikzpicture}[>=Stealth, line cap=round, line join=round, font=\footnotesize]
  \begin{scope}[shift={(-4.40,0.05)}]
    \fill[gray!6] (0,0) circle (1.25);
    \draw[red!70!black, thick] (0,0) circle (1.25);
    \draw[red!70!black, very thick, ->] (1.25,0) arc[start angle=0,end angle=65,radius=1.25];
    \node at (0,0) {$B_1$};
    \node[above] at (0,1.55) {prescribed trace on $\partial B_1$};
    \node[red!70!black, anchor=west, fill=white, inner sep=0.8pt] at (1.30,0.78) {$e^{i\theta}$};
  \end{scope}

  \draw[->, very thick] (-2.75,0.05) -- (-1.25,0.05);
  \node[above] at (-2.00,0.05) {$u$};

  \begin{scope}[shift={(1.30,0.05)}]
    \def\rr{1.62}
    \def\hh{0.64}
    \pgfmathsetmacro{\aa}{sqrt(\rr*\rr-\hh*\hh)}
    \begin{scope}
      \clip (0,0) circle (\rr);
      \fill[blue!9] (-2.0,\hh) rectangle (2.0,2.0);
      \fill[green!10] (-2.0,-2.0) rectangle (2.0,\hh);
    \end{scope}
    \draw[thick] (0,0) circle (\rr);
    \draw[red!75!black, very thick] (-\aa,\hh) -- (\aa,\hh);
    \fill (0,\rr) circle (1.1pt) node[above=2pt] {$N$};
    \fill (0,-\rr) circle (1.1pt) node[below=2pt] {$S$};
    \node[blue!70!black, anchor=west] (ulabel) at (2.00,1.04) {$u^+(B_1)$};
    \draw[->, blue!70!black] (ulabel.west) -- (0.67,1.05);
    \node[green!55!black, anchor=west] (llabel) at (2.00,-0.86) {$u^-(B_1)$};
    \draw[->, green!55!black] (llabel.west) -- (0.57,-0.86);
    \node[red!75!black, anchor=west] (blabel) at (2.00,0.31) {$g_R(\partial B_1)$};
    \draw[->, red!75!black] (blabel.west) -- (1.14,\hh);
  \end{scope}
\end{tikzpicture}
\caption{The Br\'ezis-Coron Dirichlet problem. Left: the unit disk with prescribed trace on $\partial B_1$. Right: a meridional cross-section
of the target sphere; the boundary data trace a latitude circle, and the two explicit solutions fill the corresponding upper and lower
spherical caps.}
\label{fig:problem}
\end{figure}
Our main result is the following uniqueness theorem, which answers {\it Open Problem~3.1} in \cite{Brezis2023}.

\begin{theorem}\label{thm:main}
Let $0<R<1$ and let $u\in H^1(B_1;\Sph)$ be a weak harmonic map with trace $g_R$. Then $u=u^+$ or $u=u^-$.
\end{theorem}

\begin{remark}
\begin{enumerate}
\item[(1)] Theorem~\ref{thm:main} classifies all stationary points with trace \(g_R\), not only energy minimizers. Br\'ezis and Coron
    identified two relative minimizers; the theorem rules out any non-minimizing  critical points.
\item[(2)] This is not a formal consequence of the relative-degree theory. In two dimensions, the relative degree may be lost under weak
    \(H^1\) convergence through bubbling, so a variational classification of minimizers need not control all critical points.
\item[(3)] The proof identifies a special rigidity of the once-around latitude trace. The auxiliary potential, the Pohozaev identity, and
    the sharp planar isoperimetric inequality force the Hopf differential to vanish.
\item[(4)] This contrasts with higher-winding or more general nonconvex boundary data, where the work of Br\'ezis--Coron, Soyeur, Qing,
    Kuwert, and Pierre shows that multiplicity is a natural phenomenon
    \cite{BrCo,Soyeur1989,Qing1992Remark,Qing1992Multiple,Kuwert1994,Pierre2008}.
\item[(5)] For the two endpoints of the parameter interval: $R=1$ is nondegenerate because $w_\theta(e^{i\theta})=ie^{i\theta}$, so the same
    Hopf-rigidity and collar-classification proof applies; $R=0$ collapses the boundary circle to a point, and destroys the nonvanishing
    boundary derivative used to select a holomorphic or anti-holomorphic collar.
\end{enumerate}
\end{remark}

Next, we present an application of the above uniqueness theorem for harmonic maps to the
study of \(H\)-systems. Two-dimensional \(H\)-systems belong to the broader class of
conformally invariant elliptic systems, whose conservation-law structure and regularity theory
were developed by H\'elein and, in a general form, by Rivi\`ere
\cite{Helein,Riviere2007}.
Let $\eiii=(0,0,1)$.
For $\sigma\in\{+1,-1\}$ define the latitude normal trace
\[
g_R^\sigma(\theta)
=
\bigl(R\cos\theta,\,R\sin\theta,\,\sigma\sqrt{1-R^2}\bigr),
\]
and the boundary circle
\[
\Gamma_R(\theta)=(2R\cos\theta,2R\sin\theta,0).
\]
For $\phi\in\RR$, we write
\[
\Rot_\phi(r,\theta)=(r,\theta+\phi) .
\]
The two normalized model disks with boundary $\Gamma_R(\theta)$ are
\begin{align*}
Z^+(r,\theta)
&=2u^+(r,\theta)-2\sqrt{1-R^2}e_3,\\[0.4em]
Z^-(r,\theta)
&=
-2u^-(r,\theta+\pi)+2\sqrt{1-R^2}e_3.
\end{align*}
For a conformal immersion $Y:\B\to\RR^3$, we denote its induced unit normal by
\[
\nu_Y=\frac{Y_x\times Y_y}{|Y_x\times Y_y|}.
\]
Here and below the sign of $\nu_Y$ is fixed by the ordered coordinates $(x,y)$.

\begin{corollary}\label{cor:normal-trace-Hsystem}
Let $Y\in C^2(\B;\RR^3)\cap C^1(\overline{\B};\RR^3)$ be a conformal immersion satisfying
\begin{align*}\displaystyle
\begin{cases}
        -\Delta Y=Y_x\wedge Y_y,\quad\text{in }\B.\\
        \nu_Y(1,\theta)=g_R^\sigma(\theta+\phi),
\end{cases}
\end{align*}
for some $\sigma\in\{+1,-1\}$ and some $\phi\in\RR$.
Then there exists $a\in\RR^3$ such that
\[
Y=a+Z^\sigma\circ\Rot_\phi.
\]
\end{corollary}

\begin{remark}
In the above corollary, it is sufficient to assume that
\(Y\in W^{1,2}(B_1;\mathbb R^3)\) is a weakly conformal \(H\)-system disk.
By the regularity theory for two-dimensional \(H\)-systems and, more generally,
for conformally invariant elliptic systems in the plane
\cite{Heinz1975,Gruter1981,Riviere2007}, any such weak solution is smooth in
the interior of \(B_1\). Thus the uniqueness result applies equally to weakly
conformal \(W^{1,2}\)-solutions, provided the induced unit normal has the
prescribed trace.
\end{remark}

The proof of Theorem \ref{thm:main} consists of two main steps as follows.

\smallskip
\noindent
\textit{Step 1: boundary rigidity.}
To a harmonic map $u$ we associate the auxiliary potential
\[
X_x=-u\times u_y,
\qquad
X_y=u\times u_x.
\]
This is the classical first-order potential from the Gauss map / CMC correspondence. The horizontal projection of the boundary curve of $X$
has algebraic area exactly $\pi R^2$. On the other hand, the Hopf differential and the Pohozaev identity give an upper bound on its length.
The planar isoperimetric inequality saturates and forces the pointwise boundary identities
\[
|u_r|\equiv R,
\qquad
u_r\cdot u_\theta\equiv0
\qquad\text{on }\partial B_1.
\]
This kills the boundary values of the Hopf differential.

\smallskip
\noindent
\textit{Step 2: conformal classification.}
Once the Hopf differential vanishes identically, the map is weakly conformal. In a stereographic chart near the boundary, the corresponding
complex-valued map is then Euclidean-conformal. Because the boundary derivative never vanishes, the map is either holomorphic or
anti-holomorphic on a whole collar of the boundary. A Laurent-series argument and the boundary condition give exactly $w=\rho z$ or
$w=\rho/\bar z$. Interior analyticity then extends the identification to all of $B_1$.

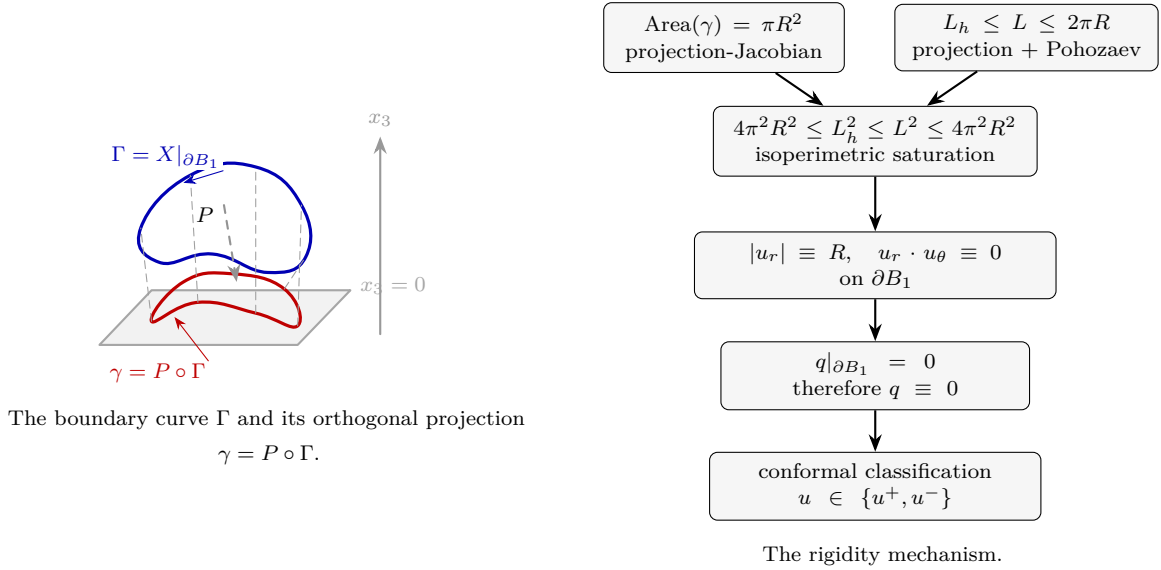
\begin{figure}[t]
\centering
\begin{minipage}{0.46\textwidth}
\centering
\begin{tikzpicture}[>=Stealth, line cap=round, line join=round, scale=0.84, font=\scriptsize]
  \coordinate (A) at (-1.90,-1.22);
  \coordinate (B) at (1.22,-1.22);
  \coordinate (C) at (2.05,-0.36);
  \coordinate (D) at (-1.08,-0.36);
  \fill[gray!10] (A)--(B)--(C)--(D)--cycle;
  \draw[gray!70, thick] (A)--(B)--(C)--(D)--cycle;
  \node[gray!70, anchor=west] at (2.05,-0.28) {$x_3=0$};
  \draw[red!75!black, very thick]
    plot[smooth cycle, tension=0.80] coordinates {(-1.08,-0.85) (-0.35,-0.55) (0.56,-0.72) (1.22,-0.92) (1.01,-0.34) (0.18,-0.10)
    (-0.58,-0.23)};
  \draw[blue!70!black, very thick]
    plot[smooth cycle, tension=0.80] coordinates {(-1.25,0.62) (-0.47,1.45) (0.56,1.58) (1.28,0.98) (1.35,0.17) (0.75,-0.08) (-0.12,0.21)
    (-0.96,0.08)};
  \draw[densely dashed, gray!75] (-1.08,-0.85) -- (-1.25,0.62);
  \draw[densely dashed, gray!75] (-0.35,-0.55) -- (-0.47,1.45);
  \draw[densely dashed, gray!75] (0.56,-0.72) -- (0.56,1.58);
  \draw[densely dashed, gray!75] (1.22,-0.92) -- (1.28,0.98);
  \draw[densely dashed, gray!75] (1.01,-0.34) -- (1.35,0.17);
  \draw[->, gray!80, thick] (2.52,-1.05) -- (2.52,2.06) node[above] {$x_3$};
  \draw[->, dashed, thick, gray!85] (0.06,0.98) -- (0.26,-0.23);
  \node[fill=white, inner sep=0.8pt] at (-0.24,0.83) {$P$};
  \node[blue!70!black, anchor=west, fill=white, inner sep=0.8pt] (Glabel) at (-1.75,1.74) {$\Gamma=X|_{\partial B_1}$};
  \draw[->, blue!70!black] (Glabel.south east) -- (-0.58,1.33);
  \node[red!75!black, anchor=west, fill=white, inner sep=0.8pt] (glabel) at (-1.78,-1.66) {$\gamma=P\circ\Gamma$};
  \draw[->, red!75!black] (glabel.north east) -- (-0.72,-0.78);
\end{tikzpicture}

\smallskip
{\scriptsize The boundary curve $\Gamma$ and its orthogonal projection $\gamma=P\circ\Gamma$.}
\end{minipage}\hfill
\begin{minipage}{0.52\textwidth}
\centering
\begin{tikzpicture}[>=Stealth,font=\scriptsize,box/.style={draw, rounded corners=3pt, align=center, inner sep=4pt,
fill=gray!7},arr/.style={->, thick}]
  \node[box, text width=2.9cm] (A) at (0.00,2.05) {$\Area(\gamma)=\pi R^2$\\projection-Jacobian};
  \node[box, text width=3.20cm] (B) at (4.00,2.05) {$L_h\le L\le2\pi R$\\projection $+$ Pohozaev};
  \node[box, text width=4.0cm] (C) at (2.00,0.68) {$4\pi^2R^2\le L_h^2\le L^2\le4\pi^2R^2$\\isoperimetric saturation};
  \node[box, text width=4.5cm] (D) at (2.00,-0.96) {$|u_r|\equiv R,\quad u_r\cdot u_\theta\equiv0$\\on $\partial B_1$};
  \node[box, text width=3.8cm] (E) at (2.00,-2.42) {$q|_{\partial B_1}=0$\\therefore $q\equiv0$};
  \node[box, text width=4.1cm] (F) at (2.00,-3.88) {conformal classification\\$u\in\{u^+,u^-\}$};
  \draw[arr] (A) -- (C);
  \draw[arr] (B) -- (C);
  \draw[arr] (C) -- (D);
  \draw[arr] (D) -- (E);
  \draw[arr] (E) -- (F);
\end{tikzpicture}

\smallskip
{\scriptsize The rigidity mechanism.}
\end{minipage}
\caption{The geometric core of the proof. The left panel shows the boundary curve $\Gamma=X|_{\partial B_1}\subset\R^3$ together with its
orthogonal projection $\gamma=P\circ\Gamma$ onto the horizontal plane $x_3=0$. The dashed segments indicate projection fibers. The right panel
summarizes the exact chain of identities and inequalities that forces the boundary vanishing of the Hopf differential; here
$L=\Length(\Gamma)$ and $L_h=\Length(\gamma)$.}
\label{fig:proof}
\end{figure}

The first step is the main new ingredient: it converts the exact algebraic area of the projected potential curve into boundary vanishing of
the Hopf differential through sharp isoperimetric saturation. It isolates a rigid geometric feature of the latitude-circle trace that is
invisible to the variational classification of the minimizing solutions. Figure~\ref{fig:proof} summarizes the geometric core of the rigidity
argument.

The paper is organized as follows. In Section~\ref{sec:prelim} we recall the variational formulation, the weak harmonic-map setting, the Hopf
differential, and the Pohozaev identity used later. In Section~\ref{sec:auxiliary} we introduce the auxiliary potential associated with a
solution and prove the projection-Jacobian identity. Section~\ref{sec:rigidity} contains the main boundary rigidity argument. In
Section~\ref{sec:conformal} we complete the conformal classification on a boundary collar by using the south-pole stereographic coordinate.
Section~\ref{sec:proof-main} combines these ingredients to prove Theorem~\ref{thm:main}, discusses the endpoint cases, and derives the
\(H\)-system consequence. Section~\ref{sec:further-remarks} records several further remarks and possible extensions. Finally,
Appendix~\ref{app:explicit} verifies the explicit Brézis--Coron solutions and their associated potentials, while Appendix~\ref{app:vectors}
collects the elementary vector identities used in the projection-Jacobian computation.

\section{Background and preliminaries}\label{sec:prelim}

\subsection{Variational origin, notation, and weak solutions}\label{subsec:variational-notation}

Let \(g\in H^{1/2}(\partial B_1;\Sph)\). The variational problem underlying the harmonic-map equation is
\begin{equation*}
\inf\left\{ E(u):=\frac12\int_{B_1}|\nabla u|^2\,dxdy: u\in\mathcal A_g\right\},
\end{equation*}
where $\mathcal A_g=\bigl\{u\in H^1(B_1;\Sph):\operatorname{Tr}_{\partial B_1}u=g\bigr\}$. The term ``harmonic map'' refers to the
Euler--Lagrange equation for this Dirichlet energy. If the target were Euclidean, the same variational principle would give \(\Delta u=0\),
the usual harmonic-function equation. For a Riemannian target, the equation is the vanishing of the tension field, i.e. the trace of the
covariant second derivative of the map \cite{EellsSampson1964,Helein}.

For the embedded target \(\Sph\subset\R^3\), the constrained Euler--Lagrange equation has a simple extrinsic form. Variations are taken with
fixed trace and under the pointwise constraint \(|u|=1\). Introducing a scalar Lagrange multiplier gives
\[
-\Delta u=\lambda u.
\]
Since \(|u|^2=1\), differentiating twice yields \(u\cdot\Delta u=-|\nabla u|^2\). Taking the dot product with \(u\) gives \(\lambda=|\nabla
u|^2\). Hence critical points with fixed trace \(g\) solve
\begin{equation}\label{eq:HM}
-\Delta u=|\nabla u|^2u
\quad\text{in }B_1,
\qquad
 u|_{\partial B_1}=g.
\end{equation}
For \(g=g_R\), this is the boundary value problem studied in the paper. Theorem~\ref{thm:main} classifies all critical points of this
constrained problem, not only its minimizers.

\paragraph{Notation.}
The following conventions are used throughout the paper.
\begin{center}
\small
\renewcommand{\arraystretch}{1.20}
\begin{tabular}{p{0.39\textwidth}p{0.53\textwidth}}
\hline
\textbf{Symbol} & \textbf{Meaning} \\
\hline
\(B_1=B_1(0)\) & unit disk in \(\R^2\) \\
\(z=x+iy=re^{i\theta}\) & Cartesian and polar coordinates \\
\(\partial_z=\frac12(\partial_x-i\partial_y)\) & complex \(z\)-derivative \\
\(\partial_{\bar z}=\frac12(\partial_x+i\partial_y)\) & complex \(\bar z\)-derivative \\
\(u=(u_1,u_2,u_3)\), \(X=(X_1,X_2,X_3)\) & component notation in \(\R^3\) \\
\(P(x_1,x_2,x_3)=(x_1,x_2)\) & horizontal projection onto \(\R^2\) \\
\(\Gamma(\theta)=X(e^{i\theta})\) & boundary curve of the potential \(X\) \\
\(\gamma=P\circ\Gamma\) & horizontal projection of \(\Gamma\) \\
\(L=\Length(\Gamma)\), \(L_h=\Length(\gamma)\) & lengths of \(\Gamma\) and \(\gamma\) \\
\(q(z)=u_z\cdot u_z\) & Hopf scalar; \(q(z)\,dz^2\) is the Hopf differential \\
\(\rho=R/(1+\sqrt{1-R^2})\) & stereographic radius of the boundary circle \\
\hline
\end{tabular}
\end{center}
Lower numerical subscripts denote components, while lower coordinate subscripts denote partial derivatives. Thus \(u_3\) is the third
component of \(u\), whereas \(u_x=\partial_xu\), \(u_r=\partial_ru\), and \(u_\theta=\partial_\theta u\). Similarly, \(\sigma_{2,\theta}\)
means \(\partial_\theta(\sigma_2)\). For a scalar function \(f\), \(\nabla f\) is its gradient. For a vector-valued map \(F\), \(DF\) denotes
its real differential, and in PDE notation we sometimes write \(\nabla F\) for the same Jacobian matrix. Hence
\[
\det\nabla(P\circ X)=\det D(P\circ X),
\qquad
\det\nabla(P\circ u)=\det D(P\circ u).
\]
All real dot and cross products are Euclidean. In the definition of the Hopf differential, the dot product is extended complex bilinearly, not
sesquilinearly.

We call $u\in H^1(B_1;\Sph)$ a weak harmonic map with trace $g_R$ if $u|_{\partial B_1}=g_R$ in the trace sense and
\[
\int_{B_1}\langle \nabla u,\nabla\varphi\rangle\,dx=\int_{B_1}|\nabla u|^2\,u\cdot\varphi\,dx
\]
for every $\varphi\in C_c^\infty(B_1;\R^3)$. Equivalently, $u$ is a distributional solution of \eqref{eq:HM} with finite Dirichlet energy and
prescribed trace.

This weak formulation is the extrinsic form of the constrained Euler--Lagrange equation. The normal term \(|\nabla u|^2u\) encodes the sphere
constraint; equivalently, if one tests only against tangent variations \(\varphi\perp u\), the Lagrange multiplier disappears and one obtains
the intrinsic tension-field formulation.

\begin{proposition}[Regularity Reduction]\label{prop:regularity}
Every weak solution of \eqref{eq:HM} is smooth up to the boundary. In particular, all boundary quantities used below (i.e. $u_r|_{\partial
B_1}$, $u_\theta|_{\partial B_1}$, the Hopf differential on $\overline{B_1}$, and the boundary curve of the auxiliary potential) are
classical, and Theorem~\ref{thm:main} can be proved in the class $C^\infty(\overline{B_1};\Sph)$.
\end{proposition}

\begin{proof}
For two-dimensional weak harmonic maps, interior smoothness follows from the regularity theory of H\'elein \cite{Helein}. Since the boundary
trace $g_R$ is smooth, boundary smoothness follows from the boundary regularity theorem of Schoen--Uhlenbeck; see also Qing's boundary
regularity theorem for weakly harmonic maps from surfaces \cite{SchoenUhlenbeckBoundary,Qing1993}. Hence every weak solution belongs to
$C^\infty(\overline{B_1};\Sph)$. The remaining assertions follow from this regularity and standard trace theory.
\end{proof}

From now on, unless explicitly stated otherwise, $u$ denotes a smooth solution of \eqref{eq:HM}.

\subsection{The Hopf differential and the Pohozaev identity}

Define the Hopf scalar
\[
q(z)=u_z\cdot u_z,
\qquad z=x+iy,
\]
where $u_z=\frac12(u_x-iu_y)$ and the dot product is extended complex bilinearly.

The Hopf differential was originally introduced in the theory of CMC surfaces: Hopf used the holomorphic quadratic differential obtained from
the \((2,0)\)-part of the second fundamental form to prove that every immersed CMC two-sphere in \(\R^3\) is round \cite{Hopf1951}. Chern
later extended this circle of ideas to space forms \cite{Chern1983}. In harmonic-map theory, the analogous object is the \((2,0)\)-part of the
pull-back metric. For maps from Riemann surfaces it is holomorphic at critical points of the Dirichlet energy, a standard fact used
systematically by Sampson and in the subsequent theory of harmonic maps from surfaces \cite{Sampson1978,EellsWood1983,Helein}. In the present
argument the Hopf differential plays two separate roles: it yields the Pohozaev identity below, and its vanishing is exactly the
weak-conformality condition that reduces the final classification to a holomorphic/anti-holomorphic alternative.

With the above convention,
\[
q=\frac14(u_x-iu_y)\cdot(u_x-iu_y)
=\frac14\bigl(|u_x|^2-|u_y|^2-2i\,u_x\cdot u_y\bigr).
\]
Consequently \(\operatorname{Re}q\) measures the difference between the two coordinate stretch factors, while \(\operatorname{Im}q\) measures
the failure of \(u_x\) and \(u_y\) to be orthogonal. Hence \(q=0\) is precisely the local weak-conformality condition \(|u_x|=|u_y|\) and
\(u_x\cdot u_y=0\).

\begin{lemma}[Holomorphicity of the Hopf Differential]\label{lem:hopf-holomorphic}
For every smooth solution of \eqref{eq:HM}, the quadratic differential $q(z)\,dz^2$ is holomorphic in $B_1$ and continuous on
$\overline{B_1}$.
\end{lemma}

\begin{proof}
Since $u$ is smooth, continuity up to $\partial B_1$ is immediate. In the interior,
\[
\partial_{\bar z}q=2u_{z\bar z}\cdot u_z=\frac12\Delta u\cdot u_z.
\]
By \eqref{eq:HM}, $\Delta u=-|\nabla u|^2u$, while differentiating $|u|^2=1$ gives $u\cdot u_z=0$. Hence $\partial_{\bar z}q=0$.
\end{proof}

In polar coordinates $z=re^{i\theta}$ one has
\begin{equation}\label{eq:Hopf-polar}
4e^{2i\theta}q
=
|u_r|^2-r^{-2}|u_\theta|^2-2ir^{-1}u_r\cdot u_\theta.
\end{equation}
Indeed,
\[
u_z=\frac12e^{-i\theta}(u_r-ir^{-1}u_\theta),
\]
and \eqref{eq:Hopf-polar} follows by direct expansion. Along the boundary,
\[
u_\theta(1,\theta)=(-R\sin\theta,R\cos\theta,0),
\qquad |u_\theta(1,\theta)|=R,
\]
so \eqref{eq:Hopf-polar} becomes
\begin{equation}\label{eq:Hopf-boundary}
4e^{2i\theta}q(e^{i\theta})
=
|u_r(1,\theta)|^2-R^2-2i\,u_r(1,\theta)\cdot u_\theta(1,\theta).
\end{equation}

\begin{lemma}[Pohozaev Identity]\label{lem:pohozaev}
Every smooth solution of \eqref{eq:HM} satisfies
\begin{equation}\label{eq:pohozaev}
\int_0^{2\pi}|u_r(1,\theta)|^2\,d\theta=2\pi R^2.
\end{equation}
\end{lemma}

\begin{proof}
By Lemma~\ref{lem:hopf-holomorphic}, $q$ is holomorphic, so $zq(z)$ is holomorphic, and therefore
\[
\int_{|z|=1}zq(z)\,dz=0.
\]
Writing $z=e^{i\theta}$ and $dz=ie^{i\theta}\,d\theta$, we get
\[
\int_0^{2\pi}e^{2i\theta}q(e^{i\theta})\,d\theta=0.
\]
Taking real parts and using \eqref{eq:Hopf-boundary} gives
\[
0=\frac14\int_0^{2\pi}\bigl(|u_r(1,\theta)|^2-R^2\bigr)\,d\theta,
\]
which is exactly \eqref{eq:pohozaev}.
\end{proof}

\section{The auxiliary potential and the projected boundary curve}\label{sec:auxiliary}

\subsection{Definition and boundary formulas}

The harmonic-map equation implies the existence of a global potential $X\colon B_1\to\R^3$, unique up to an additive constant, such that
\begin{equation}\label{eq:X-def}
X_x=-u\times u_y,
\qquad
X_y=u\times u_x.
\end{equation}
To see this directly, set
\[
A=-u\times u_y,\qquad B=u\times u_x,\qquad \omega=A\,dx+B\,dy.
\]
For a vector-valued one-form \(\omega=A\,dx+B\,dy\),
\[
d\omega=(B_x-A_y)\,dx\wedge dy.
\]
Here
\[
B_x=(u\times u_x)_x=u_x\times u_x+u\times u_{xx}=u\times u_{xx},
\]
and
\[
A_y=(-u\times u_y)_y=-u_y\times u_y-u\times u_{yy}=-u\times u_{yy}.
\]
Therefore
\[
B_x-A_y=u\times(u_{xx}+u_{yy})=u\times\Delta u.
\]
Since \(u\) solves \eqref{eq:HM}, \(\Delta u=-|\nabla u|^2u\), and hence
\[
B_x-A_y=-|\nabla u|^2(u\times u)=0.
\]
Thus \(d\omega=0\). Since
\[
H^1_{\mathrm{dR}}(B_1)=0,
\]
every closed one-form on the disk is exact. Applied componentwise to the \(\R^3\)-valued one-form \(\omega\), this gives a map \(X\colon
B_1\to\R^3\) such that \(dX=\omega\), namely \eqref{eq:X-def}. The additive constant in \(X\) is irrelevant for all boundary length and
projected-area calculations below.

The same first-order reconstruction is standard in the Gauss-map/CMC correspondence. The Ruh--Vilms theorem identifies harmonicity of the
Gauss map with parallel mean curvature vector in Euclidean submanifold theory \cite{RuhVilms1970}; in the surface case, the inverse
reconstruction formulas of Kenmotsu and related Weierstrass-type representations use a closed one-form equivalent, up to the choice of
orientation and sign convention, to \(u\times *du\) \cite{Kenmotsu1979}. Here we do not use \(X\) as an immersion; only its boundary curve and
the algebraic area of its horizontal projection enter the proof.

In polar coordinates,
\begin{equation}\label{eq:X-polar}
X_r=-\frac1r\,u\times u_\theta,
\qquad
X_\theta=r\,u\times u_r.
\end{equation}
Let
\[
\Gamma(\theta)=X(e^{i\theta}),
\qquad
\gamma(\theta)=\bigl(X_1(e^{i\theta}),X_2(e^{i\theta})\bigr)
\]
be the boundary curve of $X$ and its orthogonal projection onto the plane $x_3=0$; see the left-hand panel of Figure~\ref{fig:proof} for a
schematic.

\begin{lemma}\label{lem:length-height}
Along $\partial B_1$ one has
\begin{equation}\label{eq:Gamma-speed}
|\Gamma_\theta(\theta)|=|u_r(1,\theta)|,
\end{equation}
\begin{equation}\label{eq:X3-derivative}
\partial_\theta X_3(1,\theta)=u_r(1,\theta)\cdot u_\theta(1,\theta),
\end{equation}
and therefore
\begin{equation}\label{eq:gamma-pointwise}
|\gamma_\theta(\theta)|^2
=
|u_r(1,\theta)|^2-\bigl(u_r(1,\theta)\cdot u_\theta(1,\theta)\bigr)^2
\le |u_r(1,\theta)|^2.
\end{equation}
\end{lemma}

\begin{proof}
By \eqref{eq:X-polar},
\[
\Gamma_\theta=X_\theta=u\times u_r
\qquad\text{on }\partial B_1.
\]
Since $u\cdot u_r=0$ and $|u|=1$, this implies \eqref{eq:Gamma-speed}. For the third component,
\[
\partial_\theta X_3=(u\times u_r)\cdot e_3=u_r\cdot(e_3\times u).
\]
Along $\partial B_1$,
\[
e_3\times u(1,\theta)=(-R\sin\theta,R\cos\theta,0)=u_\theta(1,\theta),
\]
which proves \eqref{eq:X3-derivative}. Finally,
\[
|\gamma_\theta|^2=|\Gamma_\theta|^2-|\partial_\theta X_3|^2,
\]
and \eqref{eq:gamma-pointwise} follows.
\end{proof}

\subsection{A projection-Jacobian identity}

Let $P\colon\R^3\to\R^2$ be the orthogonal projection
\[
P(x_1,x_2,x_3)=(x_1,x_2).
\]

\begin{lemma}[Projection-Jacobian Identity]\label{lem:projection-jacobian}
For every $x\in B_1$,
\begin{equation}\label{eq:proj-Jac}
\det\nabla(P\circ X)=\det\nabla(P\circ u).
\end{equation}
Consequently, the projected boundary curve $\gamma$ has signed, or algebraic, area
\[
\AlgArea(\gamma)=\pi R^2.
\]
\end{lemma}

\begin{proof}
Since $\det\nabla(P\circ X)=e_3\cdot(X_x\times X_y)$, \eqref{eq:X-def} gives
\[
\det\nabla(P\circ X)
=e_3\cdot\bigl((-u\times u_y)\times(u\times u_x)\bigr).
\]
We use the elementary identity from Appendix~\ref{app:vectors}:
\begin{equation}\label{eq:vector-identity}
(a\times b)\times(a\times c)=\det(a,b,c)\,a,
\qquad a,b,c\in\R^3.
\end{equation}
Applying \eqref{eq:vector-identity} with $a=u$, $b=u_y$, and $c=u_x$, and keeping track of the minus sign, we obtain
\[
(-u\times u_y)\times(u\times u_x)=\bigl(u\cdot(u_x\times u_y)\bigr)u.
\]
Therefore
\begin{equation}\label{eq:detPX}
\det\nabla(P\circ X)=u_3\,u\cdot(u_x\times u_y).
\end{equation}
On the other hand,
\[
\det\nabla(P\circ u)=\det\nabla(u_1,u_2)=e_3\cdot(u_x\times u_y).
\]
Because $u_x,u_y\in T_u\Sph$, the vector $u_x\times u_y$ is parallel to $u$; hence
\[
e_3\cdot(u_x\times u_y)=u_3\,u\cdot(u_x\times u_y).
\]
Together with \eqref{eq:detPX}, this proves \eqref{eq:proj-Jac}.

Integrating over $B_1$ and using Green's theorem,
\[
\AlgArea(\gamma)=\int_{B_1}\det\nabla(P\circ X)\,dxdy
=\int_{B_1}\det\nabla(P\circ u)\,dxdy
=\frac12\int_{\partial B_1}(u_1\,du_2-u_2\,du_1).
\]
Since $(u_1,u_2)|_{\partial B_1}=(R\cos\theta,R\sin\theta)$ winds once positively around the circle of radius $R$, the last integral equals
$\pi R^2$.
\end{proof}

\subsection{A planar isoperimetric inequality for algebraic area}

\begin{lemma}\label{lem:isoperimetric}
Let $\sigma\colon[0,2\pi]\to\R^2$ be a closed $C^1$ curve. Define its signed, or algebraic, area by
\[
\AlgArea(\sigma)=\frac12\int_0^{2\pi}\bigl(\sigma_1\sigma_{2,\theta}-\sigma_2\sigma_{1,\theta}\bigr)\, d\theta.
\]
Then
\begin{equation}\label{eq:isoperimetric}
\Length(\sigma)^2\ge 4\pi |\AlgArea(\sigma)|.
\end{equation}
\end{lemma}

\begin{proof}
This is the algebraic-area form of the planar isoperimetric inequality for closed curves. Crucially, the curves do not need to be simple. We
cite it in the form explained by Osserman \cite{Osserman1978}: the classical Wirtinger inequality implies the isoperimetric inequality for
every smooth closed plane curve, with the area understood through the line integral defining algebraic area, and the same discussion records
the corresponding winding-number interpretation for non-simple curves \cite[\S 1, Lemmas~1.1--1.2 and (1.8)]{Osserman1978}.

For completeness, we give a sketch of proof. By $C^1$ approximation it suffices to consider smooth closed curves. Reparametrize the curve by
constant speed and translate it so that its mean over one period is zero. Then
\[
2\AlgArea(\sigma)=\int_0^{2\pi}\sigma\cdot J\sigma_\theta\,d\theta,
\qquad J(x_1,x_2)=(x_2,-x_1).
\]
Cauchy--Schwarz gives
\[
2|\AlgArea(\sigma)|\le \|\sigma\|_{L^2}\|\sigma_\theta\|_{L^2},
\]
and the periodic Wirtinger inequality applied componentwise to the mean-zero map gives
\[
\|\sigma\|_{L^2}\le \|\sigma_\theta\|_{L^2}.
\]
Since the parametrization has constant speed, $\|\sigma_\theta\|_{L^2}^2=\Length(\sigma)^2/(2\pi)$. Therefore
\[
|\AlgArea(\sigma)|\le \frac{\Length(\sigma)^2}{4\pi},
\]
which is exactly \eqref{eq:isoperimetric}. Reversing orientation if necessary handles the sign of the algebraic area.
\end{proof}

\section{Boundary rigidity}\label{sec:rigidity}

\begin{proposition}[Boundary Rigidity]\label{prop:boundary-rigidity}
Every smooth solution of \eqref{eq:HM} satisfies
\[
|u_r(1,\theta)|\equiv R,
\qquad
u_r(1,\theta)\cdot u_\theta(1,\theta)\equiv0.
\]
Equivalently,
\begin{equation}\label{eq:q-boundary-zero}
q(e^{i\theta})\equiv0
\qquad\text{on }\partial B_1.
\end{equation}
\end{proposition}

\begin{proof}
Let $X$, $\Gamma$, and $\gamma$ be as in Section~\ref{sec:auxiliary}. Set
\[
L=\Length(\Gamma)=\int_0^{2\pi}|\Gamma_\theta|\,d\theta,
\qquad
L_h=\Length(\gamma)=\int_0^{2\pi}|\gamma_\theta|\,d\theta.
\]
By Lemma~\ref{lem:length-height},
\[
L=\int_0^{2\pi}|u_r(1,\theta)|\,d\theta.
\]
Hence Cauchy--Schwarz and Lemma~\ref{lem:pohozaev} yield
\begin{equation}\label{eq:L-upper}
L^2
\le 2\pi\int_0^{2\pi}|u_r(1,\theta)|^2\,d\theta
=4\pi^2R^2.
\end{equation}
Because orthogonal projection is $1$-Lipschitz,
\begin{equation}\label{eq:Lh-proj}
L_h\le L.
\end{equation}
On the other hand, Lemma~\ref{lem:projection-jacobian} and Lemma~\ref{lem:isoperimetric} give
\begin{equation}\label{eq:Lh-lower}
L_h^2\ge 4\pi |\AlgArea(\gamma)|=4\pi^2R^2.
\end{equation}
Combining \eqref{eq:L-upper}, \eqref{eq:Lh-proj}, and \eqref{eq:Lh-lower}, we obtain
\[
4\pi^2R^2\le L_h^2\le L^2\le 4\pi^2R^2,
\]
so all inequalities are equalities and
\[
L_h=L=2\pi R.
\]
This is the only point in the proof where sharpness is used. The lower bound comes from the exact algebraic area of the projected curve, while
the upper bound comes from the exact Pohozaev identity. Since the two sharp constants coincide, there is no room for strict inequality at any
intermediate step.

Equality in Cauchy--Schwarz in \eqref{eq:L-upper} implies that $|u_r(1,\theta)|$ is constant in $\theta$. Together with
Lemma~\ref{lem:pohozaev}, this gives
\begin{equation}\label{eq:ur-constant}
|u_r(1,\theta)|\equiv R.
\end{equation}
Next, \eqref{eq:gamma-pointwise} gives the pointwise estimate $|\gamma_\theta|\le |\Gamma_\theta|$. Since the nonnegative continuous function
$|\Gamma_\theta|-|\gamma_\theta|$ integrates to $L-L_h=0$, it vanishes identically. Therefore
\[
|\gamma_\theta|\equiv |\Gamma_\theta|,
\]
and \eqref{eq:gamma-pointwise} implies
\begin{equation}\label{eq:boundary-orthogonality}
 u_r(1,\theta)\cdot u_\theta(1,\theta)\equiv0.
\end{equation}
Finally, substituting \eqref{eq:ur-constant} and \eqref{eq:boundary-orthogonality} into \eqref{eq:Hopf-boundary} yields
\eqref{eq:q-boundary-zero}.
\end{proof}

\begin{remark}
The proof uses only the equality cases that are needed for the Hopf differential: equality in Cauchy--Schwarz gives $|u_r|\equiv R$, and
equality in the projection-length inequality gives $u_r\cdot u_\theta\equiv0$. No global embeddedness of $\gamma$ is assumed. In particular,
the isoperimetric inequality is applied to algebraic area, so possible self-intersections would not affect the argument. If one also invokes
the equality case in the Wirtinger proof of Lemma~\ref{lem:isoperimetric}, then $\gamma$ is an equi-speed round circle; together with
$\partial_\theta X_3=0$ this recovers the horizontal-circle picture verified explicitly in Appendix~\ref{app:explicit}.
\end{remark}

\section{Conformal classification on a boundary collar}\label{sec:conformal}

In this section we classify the solutions under the additional assumption $q\equiv0$. This is the usual weak-conformal harmonic-map dichotomy
in a form tailored to the present boundary trace; compare the standard meromorphic/anti-meromorphic classification for conformal harmonic maps
into $\mathbb{CP}^1$ \cite{EellsWood1983,Helein}. The argument is local near the boundary at first: the boundary trace avoids the south pole,
so the south stereographic coordinate is defined on a collar annulus. Figure~\ref{fig:collar} illustrates this reduction and the two possible
orientations of the conformal coordinate map.

\begin{figure}[t]
\centering
\begin{tikzpicture}[>=Stealth, line cap=round, line join=round, font=\scriptsize]
  \begin{scope}[shift={(-3.85,0)}, scale=1.02]
    \fill[blue!8] (0,0) circle (1.68);
    \fill[white] (0,0) circle (1.08);
    \draw[red!70!black, thick] (0,0) circle (1.68);
    \draw[red!70!black, thick, dashed] (0,0) circle (1.08);
    \draw[red!70!black, very thick, ->] (1.68,0) arc[start angle=0,end angle=65,radius=1.68];
    \node at (0,0) {$B_1$};
    \node[blue!70!black, fill=white, inner sep=0.8pt] at (0,-1.36) {$A=\{r_1<|z|<1\}$};
    \node[fill=white, inner sep=0.8pt] at (-0.55,1.92) {$|z|=1$};
    \node[red!70!black, anchor=west, fill=white, inner sep=0.8pt] (bd) at (1.12,0.90) {$w(e^{i\theta})=\rho e^{i\theta}$};
  \end{scope}
  \draw[->, very thick] (-1.38,0.08) -- (0.26,0.08);
  \node[above] at (-0.56,0.08) {$w=\pi_S\circ u$};
  \begin{scope}[shift={(2.95,0)}, scale=1.02]
    \draw[->] (-1.72,0) -- (1.92,0) node[right] {$\Re w$};
    \draw[->] (0,-1.50) -- (0,1.72) node[above] {$\Im w$};
    \draw[red!70!black, very thick] (0,0) circle (1.05);
    \draw[->, blue!75!black, very thick] (1.05,0) arc[start angle=0,end angle=56,radius=1.05];
    \draw[->, green!55!black, very thick] (0.55,-0.89) arc[start angle=-58,end angle=-114,radius=1.05];
    \node[red!70!black, anchor=center, fill=white, inner sep=0.8pt] (circlelabel) at (0.91,-0.55) {$|w|=\rho$};
    \node[blue!75!black, anchor=east, fill=white, inner sep=0.8pt] (pos) at (-0.70,1.48) {$J_w>0:\; w=\rho z$};
    \draw[->, blue!75!black] (pos.south east) -- (0.93,0.49);
    \node[green!55!black, anchor=east, fill=white, inner sep=0.8pt] (neg) at (-0.70,-1.40) {$J_w<0:\; w=\rho/\bar z$};
    \draw[->, green!55!black] (neg.north east) -- (0.07,-1.05);
  \end{scope}
\end{tikzpicture}
\caption{The collar classification in Section~\ref{sec:conformal}. Since the boundary trace avoids the south pole, stereographic projection
gives a smooth complex-valued map \(w\) on a collar annulus \(A\). The vanishing of the Hopf differential makes \(w\) conformal there. The
nonvanishing boundary derivative fixes the sign of the real Jacobian \(J_w=\det Dw\) on a smaller collar, giving either the holomorphic branch
\(w=\rho z\) or the anti-holomorphic branch \(w=\rho/\bar z\).}
\label{fig:collar}
\end{figure}
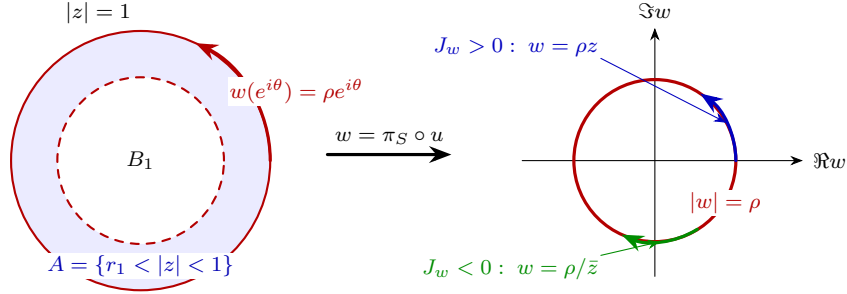

\begin{proposition}\label{prop:conformal-classification}
Let $0<R\le1$ and let $u\in C^\infty(\overline{B_1};\Sph)$ be a solution of \eqref{eq:HM} with trace $g_R$. If $q\equiv0$ in $B_1$, then
$u=u^+$ or $u=u^-$.
\end{proposition}

\begin{proof}
Let $S=(0,0,-1)$ be the south pole. Since the boundary trace $g_R$ stays a positive distance away from $S$ and $u$ is continuous on
$\overline{B_1}$, we can choose $r_0<1$ sufficiently close to $1$ such that
\[
u\bigl(\{r_0\le |z|\le 1\}\bigr)\subset \Sph\setminus\{S\}.
\]
On this closed collar we may use the south-pole stereographic coordinate $\pi_S$ introduced above. Set
\[
w=\pi_S\circ u.
\]
Choose $r_1\in(r_0,1)$, to be fixed below, and set
\[
A=\{z\in\C:r_1<|z|<1\},
\qquad
\overline A=\{z\in\C:r_1\le |z|\le1\}.
\]
Then $w\in C^\infty(\overline A)$, and the boundary condition becomes
\begin{equation}\label{eq:w-boundary}
 w(e^{i\theta})=\rho e^{i\theta},
\end{equation}
where $\rho=R/(1+\sqrt{1-R^2})$.

In the stereographic coordinate $w$, the round metric on $\Sph\setminus\{S\}$ is conformal to the Euclidean metric:
\[
g_{\Sph}=\lambda(w)^2|dw|^2,
\qquad
\lambda(w)=\frac{2}{1+|w|^2}.
\]
Hence
\[
q=\langle u_z,u_z\rangle_{\Sph}=\lambda(w)^2 w_z\overline{w_{\bar z}}.
\]
Since $q\equiv0$, we have $w_z\overline{w_{\bar z}}\equiv0$ in the open collar. Equivalently, in real coordinates,
\begin{equation}\label{eq:euclidean-conformal}
|w_x|=|w_y|,
\qquad
w_x\cdot w_y=0.
\end{equation}
Thus $w$ is Euclidean-conformal in the open annulus.

From \eqref{eq:w-boundary},
\[
w_\theta(e^{i\theta})=i\rho e^{i\theta}\neq0.
\]
Taking $r_1$ sufficiently close to $1$, the smoothness of $w$ on $\overline A$ gives
\[
|w_\theta(z)|\ge \frac\rho2
\qquad\text{for every }z\in\overline A.
\]
In particular, $Dw\neq0$ in $A$. Let $J_w=\det Dw$ be the real Jacobian determinant of the map $w\colon A\subset\R^2\to\R^2$. By
\eqref{eq:euclidean-conformal},
\[
J_w^2=|w_x|^4.
\]
Since $Dw\neq0$, this shows that $J_w$ never vanishes in $A$. The annulus $A$ is connected, so the sign of $J_w$ is constant.

Suppose first that $J_w>0$ in $A$. Then the vectors $w_x$ and $w_y$ are orthogonal, have the same length, and form a positively oriented pair.
Hence
\[
w_y=iw_x
\qquad\text{in }A.
\]
Thus $w$ is holomorphic in the open annulus $A$. We emphasize that no holomorphic extension across $|z|=1$ is being used. The Laurent
expansion is only taken in $A$:
\[
w(z)=\sum_{n\in\mathbb Z}a_nz^n,
\qquad r_1<|z|<1.
\]
For $r_1<r<1$, the Laurent coefficient formula gives
\[
a_n=\frac{1}{2\pi i}\int_{|\zeta|=r}w(\zeta)\zeta^{-n-1}\,d\zeta
=\frac{r^{-n}}{2\pi}\int_0^{2\pi}w(re^{i\theta})e^{-in\theta}\,d\theta.
\]
Since $w\in C^\infty(\overline A)$, the convergence
\[
w(re^{i\theta})\longrightarrow w(e^{i\theta})
\qquad\text{as }r\uparrow1
\]
is uniform in $\theta$. Letting $r\uparrow1$ in the preceding formula and using \eqref{eq:w-boundary}, we obtain
\[
a_n=\frac1{2\pi}\int_0^{2\pi}\rho e^{i\theta}e^{-in\theta}\,d\theta
=\rho\delta_{n1}.
\]
Therefore
\[
w(z)=\rho z
\qquad\text{in }A,
\]
and hence $u=u^+$ in $A$.

Suppose instead that $J_w<0$ in $A$. Then the pair $(w_x,w_y)$ is negatively oriented, and \eqref{eq:euclidean-conformal} gives
\[
w_y=-iw_x
\qquad\text{in }A.
\]
Thus $w$ is anti-holomorphic in $A$. Equivalently, $h=\overline w$ is holomorphic in $A$ and smooth on $\overline A$. On the outer boundary,
\[
h(e^{i\theta})=\overline{w(e^{i\theta})}=\rho e^{-i\theta}.
\]
Writing
\[
h(z)=\sum_{n\in\mathbb Z}b_nz^n,
\qquad r_1<|z|<1,
\]
we have, for $r_1<r<1$,
\[
b_n=\frac{1}{2\pi i}\int_{|\zeta|=r}h(\zeta)\zeta^{-n-1}\,d\zeta
=\frac{r^{-n}}{2\pi}\int_0^{2\pi}h(re^{i\theta})e^{-in\theta}\,d\theta.
\]
Since $h\in C^\infty(\overline A)$, we may let $r\uparrow1$ and use $h(e^{i\theta})=\rho e^{-i\theta}$. Thus
\[
b_n=\frac1{2\pi}\int_0^{2\pi}\rho e^{-i\theta}e^{-in\theta}\,d\theta
=\rho\delta_{n,-1}.
\]
Consequently
\[
h(z)=\rho z^{-1}
\qquad\text{in }A,
\]
and therefore
\[
w(z)=\frac\rho{\bar z}
\qquad\text{in }A.
\]
Thus $u=u^-$ in $A$.

It remains to pass from the collar $A$ to the whole disk. Smooth solutions of the semilinear elliptic system \eqref{eq:HM} are real analytic
in the interior of $B_1$ by standard analytic elliptic regularity \cite{Morrey1958}. Since $u$ agrees with one of the analytic maps $u^\pm$ on
the nonempty open annulus $A$, the identity theorem for real-analytic functions implies that the same equality holds throughout $B_1$.
\end{proof}

\begin{remark}
The collar classification uses only holomorphicity or anti-holomorphicity inside the annulus and the smooth trace on the outer boundary. It
does not require analytic continuation of $w$ across $\partial B_1$, and it does not assume that the stereographic coordinate is defined at
the center of the disk.
\end{remark}

\begin{remark}
The formula $w=\rho/\bar z$ has a pole at $z=0$ when written in the stereographic coordinate from the south pole. This pole is not a
singularity of the corresponding sphere-valued map: it represents the value $S=(0,0,-1)$ of $u^-$. The collar argument above avoids any
ambiguity by using the finite coordinate $w$ only near $\partial B_1$, where the trace stays a positive distance from $S$, and then extends
the resulting identity for $u$ by real analyticity.
\end{remark}

\section{Proof of the main theorem}\label{sec:proof-main}
\subsection{Proof of Theorem~\ref{thm:main}}
\begin{proof}[Proof of Theorem~\ref{thm:main}]
By Proposition~\ref{prop:regularity}, the weak solution $u$ is smooth up to the boundary. Proposition~\ref{prop:boundary-rigidity} shows that
$q$ vanishes on $\partial B_1$. Since $q$ is holomorphic in $B_1$ and continuous on $\overline{B_1}$, the maximum modulus principle gives
\[
q\equiv0
\qquad\text{in }B_1.
\]
Proposition~\ref{prop:conformal-classification} now yields $u=u^+$ or $u=u^-$. For the direct verification of these explicit formulas and
their associated potentials, see Appendix~\ref{app:explicit}. This proves the theorem.
\end{proof}

The proof of Theorem~\ref{thm:main} includes the equatorial endpoint $R=1$. In this case
\[
g_1(e^{i\theta})=(\cos\theta,\sin\theta,0),
\qquad \rho=1,
\]
and the two explicit maps become
\[
u^+(r,\theta)=\left(
\frac{2r}{1+r^2}\cos\theta,
\frac{2r}{1+r^2}\sin\theta,
\frac{1-r^2}{1+r^2}
\right),
\]
\[
u^-(r,\theta)=\left(
\frac{2r}{1+r^2}\cos\theta,
\frac{2r}{1+r^2}\sin\theta,
\frac{r^2-1}{1+r^2}
\right).
\]
They fill the northern and southern hemispheres, respectively, and satisfy the same equatorial trace. None of the rigidity estimates
degenerates at $R=1$ as  the Pohozaev identity gives
\[
\int_0^{2\pi}|u_r(1,\theta)|^2\, d\theta=2\pi,
\]
the projection-Jacobian identity gives
\[
\AlgArea(\gamma)=\pi,
\]
and the sharp chain
\[
4\pi^2\le L_h^2\le L^2\le 4\pi^2
\]
still forces
\[
|u_r(1,\theta)|\equiv1,
\qquad
u_r(1,\theta)\cdot u_\theta(1,\theta)\equiv0.
\]
Moreover the collar classification remains nondegenerate because
\[
w(e^{i\theta})=e^{i\theta},
\qquad
w_\theta(e^{i\theta})=ie^{i\theta}\ne0.
\]
Thus the two maps displayed above are the only weak harmonic maps with equatorial once-around trace.

\begin{remark}[The degenerate endpoint $R=0$]
At $R=0$ the boundary data collapse to the constant north pole,
\[
g_0(e^{i\theta})=(0,0,1)=:N,
\qquad \rho=0.
\]
The correct conclusion is no longer a two-map statement: the only smooth harmonic map with this trace is the constant map $u\equiv N$. Indeed,
the Pohozaev identity gives
\[
\int_0^{2\pi}|u_r(1,\theta)|^2\, d\theta=0,
\]
hence $u_r=0$ on $\partial B_1$, while the constant trace gives $u_\theta=0$ there. Thus $u$ has the same boundary value and normal derivative
as the constant solution. Boundary unique continuation for analytic semilinear elliptic systems then gives $u\equiv N$.

This endpoint is qualitatively different from $R=1$. For $R=1$ one still has $\rho=1$ and a nonzero boundary derivative
$w_\theta(e^{i\theta})=ie^{i\theta}$, so the collar orientation argument used in Proposition~\ref{prop:conformal-classification} remains
valid. For $R=0$, however, $w(e^{i\theta})=0$ and $w_\theta(e^{i\theta})=0$, so the step that selects a holomorphic or anti-holomorphic branch
from the sign of the Jacobian loses its nondegenerate input.

The degeneration as $R\downarrow0$ reflects this collapse. Since $\rho\downarrow0$, the upper branch satisfies $u^+\to N$ strongly, and its
conformal energy equals the area of the small northern cap,
\[
E(u^+)=2\pi(1-\sqrt{1-R^2})\to0.
\]
The lower branch converges to $N$ locally away from the origin but carries one bubble at the origin: for every $\rho>0$, $u^-(0)=S$, and
\[
E(u^-)=2\pi(1+\sqrt{1-R^2})\to4\pi.
\]
At the level of the rigidity mechanism, the projected algebraic area $\pi R^2$, the Pohozaev boundary mass $2\pi R^2$, and the length scale
$2\pi R$ all tend to zero. The sharp inequality chain remains formally true, but it collapses to $0\le0\le0$ and no longer supplies the
nonzero boundary speed needed for the collar classification. Thus $R=0$ is a genuine degeneration of the two-branch picture, whereas $R=1$ is
a nondegenerate endpoint included in Theorem~\ref{thm:main}.
\end{remark}

\medskip
\noindent\textbf{Energy and min-max picture.}
Let $c=\sqrt{1-R^2}$. The exact energy levels are
\[
 E(u^+)=2\pi(1-c),\qquad E(u^-)=2\pi(1+c),
\]
so $E(u^+)<E(u^-)$ for $0<R<1$, $E(u^+)=E(u^-)=2\pi$ for $R=1$, and
\[
 E(u^+)+E(u^-)=4\pi,
 \qquad
 E(u^-)-E(u^+)=4\pi\sqrt{1-R^2}.
\]
The proof is just the conformal area formula: $u^+$ fills the spherical cap $\{x_3\ge c\}$ and $u^-$ fills $\{x_3\le c\}$.

Thus, a finite-dimensional mountain-pass picture is misleading. A strongly convergent Palais--Smale min-max sequence would limit to a harmonic
map with trace $g_R$, hence by Theorem~\ref{thm:main} to $u^+$ or $u^-$. Therefore a third saddle cannot occur; the only possible transition
between the two relative classes is noncompactness by bubbling.
As $R\downarrow0$, $u^+\to N$ strongly, while $u^-\to N$ locally away from $0$ and carries one $4\pi$ bubble at the origin.

\subsection{Proof of Corollary~\ref{cor:normal-trace-Hsystem}}

\begin{lemma}\label{reverseimpossible}
Let $Y:B_1\to\mathbb R^3$ be a smooth conformal immersion satisfying the
$H$-system
\[
 -\Delta Y=Y_x\wedge  Y_y \quad\text{in }B_1.
\]
Let $\nu_Y$
be the unit normal induced by the orientation of the parameter disk.  Then  $\nu_Y$ cannot be
 the orientation-reversing alternative, i.e.
\[
 \nu_Y=u^-\circ\mathcal R_\phi.
\]
\end{lemma}

\begin{proof}
We first recall the geometric meaning of the shape operator.  Since $Y$ is a
conformal immersion, there exists a function $\lambda$ such that
\[
 Y_x\cdot Y_y=0,
 \qquad
 |Y_x|^2=|Y_y|^2=e^{2\lambda}.
\]
The orientation of the immersed disk is fixed by
\[
 \nu_Y=\frac{Y_x\times Y_y}{e^{2\lambda}}.
\]
The shape operator $A$ of the immersion, with respect to this normal, is the
bundle endomorphism of the tangent bundle defined by the Weingarten equation
\[
 d\nu_Y(V)=-dY(AV)
 \qquad\text{for every tangent vector }V.
\]
Equivalently, $A$ is the differential of the Gauss map written back on the
tangent plane of the surface.  Its eigenvalues $\kappa_1,\kappa_2$ are the
principal curvatures, and with this convention the scalar mean curvature is
\[
 H=\frac{\kappa_1+\kappa_2}{2}.
\]
This is why the shape operator is useful here: it converts the conformality and
orientation of the Gauss map $\nu_Y$ into algebraic information about the
principal curvatures of the surface.
With the above convention one has the standard identity, in conformal
coordinates,
\[
 \Delta Y=2e^{2\lambda}H\nu_Y.
\]
On the other hand,
\[
 Y_x\times Y_y=e^{2\lambda}\nu_Y.
\]
Therefore, the $H$-system gives
\[
 -2e^{2\lambda}H\nu_Y=e^{2\lambda}\nu_Y,
\]
and hence
\[
 H=-\frac12.
\]
In particular, the scalar mean curvature is nonzero everywhere.

We now use the orientation-reversing assumption.  The explicit Br\'ezis--Coron
alternative $u^-$ is anti-conformal as a map into $S^2$; equivalently, away
from its possible isolated branch points,
\[
 (u^-)_x\cdot (u^-)_y=0,
 \qquad
 |(u^-)_x|=|(u^-)_y|,
 \qquad
 u^-\cdot\bigl((u^-)_x\times (u^-)_y\bigr)<0.
\]
Precomposition by the rotation $\mathcal R_\phi$ preserves this sign.  Thus,
if $\nu_Y=u^-\circ\mathcal R_\phi$, then the Gauss map $\nu_Y$ is
anti-conformal and orientation-reversing, namely
\[
 |(\nu_Y)_x|=|(\nu_Y)_y|,
 \qquad
 (\nu_Y)_x\cdot(\nu_Y)_y=0,
 \qquad
 J_{\nu_Y}:=\nu_Y\cdot\bigl((\nu_Y)_x\times(\nu_Y)_y\bigr)<0.
\]

We relate this to the shape operator.  From the Weingarten equation,
$d\nu_Y=-dY\circ A$.  Since $dY$ is conformal with conformal factor
$e^\lambda$, the pull-back metric of the Gauss map satisfies
\[
 \nu_Y^*g_{S^2}(V,W)=e^{2\lambda}\,\langle AV,AW\rangle
\]
for tangent vectors $V,W$ on $B_1$.  Therefore the conformality of $\nu_Y$
implies that $A^2$ is a scalar multiple of the identity.  Because $A$ is
self-adjoint, this means that the principal curvatures satisfy
\[
 \kappa_1^2=\kappa_2^2.
\]
Moreover,
\[
 J_{\nu_Y}
 =\nu_Y\cdot\bigl((\nu_Y)_x\times(\nu_Y)_y\bigr)
 =e^{2\lambda}\det A
 =e^{2\lambda}\kappa_1\kappa_2.
\]
The orientation-reversing condition $J_{\nu_Y}<0$ therefore implies
\[
 \kappa_1\kappa_2<0.
\]
Combining this with $\kappa_1^2=\kappa_2^2$ gives
\[
 \kappa_2=-\kappa_1.
\]
Consequently
\[
 H=\frac{\kappa_1+\kappa_2}{2}=0.
\]
This contradicts the value $H=-\frac12$ forced by the $H$-system orientation.
Hence the orientation-reversing alternative cannot occur.
\end{proof}
\begin{proof}[Proof of Corollary~\ref{cor:normal-trace-Hsystem}]
We give the proof in three steps.

\smallskip
\noindent\textbf{Step 1: the induced normal is a harmonic map.}

Write the conformal factor as
\[
|Y_x|^2=|Y_y|^2=e^{2\lambda},
\qquad
Y_x\cdot Y_y=0.
\]
Since $Y$ is oriented by $(x,y)$,
\[
Y_x\wedge Y_y=e^{2\lambda}\nu_Y.
\]
For a conformal immersion into $\RR^3$, the mean-curvature scalar $H$ with respect to $\nu_Y$ satisfies
\[
\Delta Y=2e^{2\lambda}H\nu_Y.
\]
As in the proof of Lemma \ref{reverseimpossible}, the $H$-system equation gives $H=-1/2$.
In particular, the mean curvature is constant. The standard Gauss-map equation for a conformal immersion is
\[
\Delta_g\nu_Y+|A|_g^2\nu_Y=\nabla_g H,
\]
where $g=e^{2\lambda}(dx^2+dy^2)$ is the induced metric and $A$ is the second fundamental form. Since $H$ is constant, the right-hand side
vanishes. Multiplying by $e^{2\lambda}$ gives
\[
\Delta \nu_Y+|\nabla \nu_Y|^2\nu_Y=0.
\]
Thus $\nu_Y$ is a weak harmonic map from $\B$ into $\Sph$.

\smallskip
\noindent\textbf{Step 2: apply the Br\'ezis-Coron uniqueness theorem to the normal.}

(1) First, we suppose $\sigma=+1$. Define
\[
\widetilde\nu(r,\theta)=\nu_Y(r,\theta-\phi).
\]
Then $\widetilde\nu$ is a harmonic map into $\Sph$ and
\[
\widetilde\nu(1,\theta)=g_R^+(\theta).
\]
By Theorem~\ref{thm:main}, we have
\[
\widetilde\nu=u^+\quad\text{or}\quad \widetilde\nu=u^-.
\]
A direct computation gives
\[
u^+\cdot(u^+_x\times u^+_y)
=
\frac{4\rho^2}{(1+\rho^2r^2)^2}>0,
\qquad
u^-\cdot(u^-_x\times u^-_y)
=
-\frac{4\rho^2}{(\rho^2+r^2)^2}<0.
\]
Thus $u^+$ is conformal and orientation-preserving, whereas $u^-$ is conformal and orientation-reversing.

Lemma \ref{reverseimpossible} shows that the orientation-reversing alternative is incompatible with the $H$-system orientation. This excludes
$u^-$ in the case $\sigma=+1$, and gives
\[
\nu_Y=u^+\circ\Rot_\phi.
\]
Similar to Lemma \ref{reverseimpossible}, we have $H=-1/2$, $A=-I/2$. Then from the Weingarten formula $d\nu =-dY\circ A$, we have
\[ d\nu =\frac12 dY, \quad d Y=2 d \nu. \]
Since $B_1$ is connected, we have
\[
Y=a+2u^+\circ\Rot_\phi.
\]
Since $Z^+=2u^+-2\sqrt{1-R^2}\,\eiii$, the constant $2\sqrt{1-R^2}\,\eiii$ can be absorbed into $a$, and therefore
\[
Y=a+Z^+\circ\Rot_\phi.
\]

\smallskip
(2) Now suppose $\sigma=-1$. Set
\[
\widetilde\nu(r,\theta)=-\nu_Y(r,\theta-\phi+\pi).
\]
Then $\widetilde\nu$ is again a harmonic map into $\Sph$, and
\[
\widetilde\nu(1,\theta)
=-g_R^-(\theta+\pi)
=g_R^+(\theta).
\]
By Theorem~\ref{thm:main}, we have
\[
\widetilde\nu=u^+\quad\text{or}\quad \widetilde\nu=u^-.
\]
The first alternative would make $\nu_Y$ orientation-reversing, because the antipodal map on $\Sph$ reverses orientation. This is impossible
by the shape-operator argument as in Lemma \ref{reverseimpossible}. Therefore
\[
\widetilde\nu=u^-.
\]
Equivalently,
\[
\nu_Y=-u^-\circ\Rot_{\phi+\pi}.
\]
Hence, using again $Y=2\nu_Y+a$,
\[
Y=a-2u^-\circ\Rot_{\phi+\pi}.
\]
By the definition of $Z^-$,
\[
Z^-= -2u^-\circ\Rot_\pi+2\sqrt{1-R^2}\,\eiii,
\]
and the vertical translation is absorbed into $a$. Thus
\[
Y=a+Z^-\circ\Rot_\phi.
\]

\end{proof}

The normal-trace condition in Corollary~\ref{cor:normal-trace-Hsystem} can be guaranteed by several familiar types of boundary data. We record
them in the normalized case $\phi=0$; the rotated case is obtained by replacing $\theta$ with $\theta+\phi$.

\begin{itemize}
\item\textbf{Circular Dirichlet boundary plus constant contact angle with the horizontal plane.}
Assume
\begin{align*}
    \begin{cases}
        Y(1,\theta)=a+\Gamma_R(\theta),\\
        \nu_Y(1,\theta)\cdot e_3=\sigma\sqrt{1-R^2},\\
\nu_Y(1,\theta)\cdot (\cos\theta,\sin\theta,0)>0.
    \end{cases}
\end{align*}
Since $\nu_Y$ is orthogonal to the boundary tangent
\[
\partial_\theta \Gamma_R(\theta)=(-2R\sin\theta,2R\cos\theta,0),
\]
its horizontal component must be parallel to $(\cos\theta,\sin\theta,0)$. The unit-length condition and the prescribed vertical component
then force
\[
\nu_Y(1,\theta)=g_R^\sigma(\theta).
\]
\item \textbf{Tangency to the comparison sphere.}
Assume
\[
Y(1,\theta)=a+\Gamma_R(\theta)
\]
and assume that, along the boundary, the tangent plane of $Y$ agrees with the tangent plane of the sphere
\[
\left|x-\bigl(a-2\sigma\sqrt{1-R^2}\,e_3\bigr)\right|=2,
\]
with the same oriented unit normal. At the boundary point
\[
a+\Gamma_R(\theta),
\]
the outward unit normal of this sphere is
\[
\frac{a+\Gamma_R(\theta)-\bigl(a-2\sigma\sqrt{1-R^2}\,e_3\bigr)}{2}
=
\bigl(R\cos\theta,R\sin\theta,\sigma\sqrt{1-R^2}\bigr).
\]
Therefore
\[
\nu_Y(1,\theta)=g_R^\sigma(\theta).
\]

\item \textbf{Circular boundary on the vertical cylinder with prescribed angle.}
Let
\[
e_r(\theta)=(\cos\theta,\sin\theta,0).
\]
Assume the boundary is the horizontal circle
\begin{align*}
    \begin{cases}
        Y(1,\theta)=a+\Gamma_R(\theta)\\
        \nu_Y(1,\theta)\cdot e_r(\theta)=R,
\\
\operatorname{sign}\bigl(\nu_Y(1,\theta)\cdot e_3\bigr)=\sigma.
    \end{cases}
\end{align*}
Because $\nu_Y$ is unit length and orthogonal to the circular tangent direction, it has the form
\[
\nu_Y(1,\theta)=R e_r(\theta)+c(\theta)e_3.
\]
The unit-length condition gives
\[
|c(\theta)|=\sqrt{1-R^2},
\]
and the prescribed sign gives $c(\theta)=\sigma\sqrt{1-R^2}$. Hence
\[
\nu_Y(1,\theta)=g_R^\sigma(\theta).
\]
\end{itemize}

\begin{remark}
All conditions above are stronger than the normal-trace condition, but they are often easier to state from the outside: they use a boundary
curve, a contact angle, a conormal direction, or tangency to a comparison sphere. By contrast, a pure Dirichlet condition such as
$Y|_{\partial\B}=\Gamma_R$ or $Y|_{\partial\B}=0$ does not determine $\nu_Y|_{\partial\B}$ and therefore does not imply
Corollary~\ref{cor:normal-trace-Hsystem} by itself.
\end{remark}

\section{Further remarks}\label{sec:further-remarks}

Up to rigid motions of $\Sph$ and rigid rotations of $\partial B_1$, Theorem~\ref{thm:main} treats the model case of a constant-speed,
once-around parametrization of a round circle. Several nearby uniqueness problems are therefore natural. Even for the same circle but a
non-constant-speed parametrization $g$, the projection-Jacobian identity still fixes the algebraic area of the projected $X$-curve, whereas
the boundary term in the Pohozaev identity becomes $\int_0^{2\pi}|g'(\theta)|^2\,d\theta$; the exact saturation mechanism summarized in
Figure~\ref{fig:proof} is then lost.

A second family is given by the higher-winding traces
\[
 g_{R,m}(e^{i\theta})=
 \bigl(R\cos(m\theta),R\sin(m\theta),\sqrt{1-R^2}\bigr),
 \qquad m\ge 2.
\]
Here the branched conformal maps $w(z)=\rho z^m$ and $w(z)=\rho/\bar z^m$ already provide two natural harmonic maps with that trace, so one
expects a subtler classification problem. For these higher-winding data, the boundary carries more relative degree. The work of Soyeur, Qing,
and Kuwert shows that the minimizing theory then has several attainable relative classes, rather than only the two endpoint classes visible in
the once-around case \cite{Soyeur1989,Qing1992Remark,Qing1992Multiple,Kuwert1994}. This is the topological mechanism behind the expected
multiplicity: different portions of the boundary winding can be filled by meromorphic and anti-meromorphic pieces, and weak minimizing
sequences may also lose degree by bubbling. The present theorem says that this mechanism collapses completely when \(m=1\): the two endpoint
fillings survive, but no additional saddle or mountain-pass critical point is compatible with the Hopf-differential rigidity. More generally,
for non-round Jordan curves not contained in a geodesically convex ball, it is natural to ask whether some replacement for the exact area
identity and the sharp length bound can still force the boundary vanishing of the Hopf differential. By contrast, when the image of the trace
lies in a geodesically convex ball, uniqueness is classical \cite{HKW1977,JaegerKaul1979}. We hope that the boundary-rigidity viewpoint
developed here may also be useful in such nonconvex boundary-value problems.

The preceding questions are also close in spirit to Hopf-type rigidity and integrability phenomena in two-dimensional geometry and dynamics.
Bialy's work on convex billiards and Hopf rigidity shows how a sharp integral inequality can force a geometric object to be round
\cite{Bialy1993,Bialy2013,Bialy2015}. Taimanov's work on topological obstructions to integrability and on closed extremals gives complementary
examples of how global topological or variational information constrains geodesic and Lagrangian dynamics on surfaces
\cite{Taimanov1988,Taimanov1992}. It is also useful to compare the role of the auxiliary potential $X$ with Taimanov's Gauss-map and spinor
approach to surface theory: in the global Weierstrass representation of closed surfaces and in the two-dimensional Dirac-operator formalism,
Gauss-map data and first-order Dirac systems organize the reconstruction and deformation of surfaces
\cite{Taimanov1998Weierstrass,Taimanov2006Dirac}. The present argument is different in substance: the fixed quantity is the algebraic area of
the projected boundary curve associated with $X$. Nevertheless, the equality-forcing mechanism in Proposition~\ref{prop:boundary-rigidity}
belongs to this broader circle of rigidity ideas.

\appendix

\section{The explicit solutions and their potentials}\label{app:explicit}

For completeness we record a direct check for the two explicit maps.

\begin{proposition}\label{prop:explicit-check}
The maps $u^\pm$ defined in the Introduction solve \eqref{eq:HM}. Moreover, one may choose the associated potentials to be
\[
X^+=u^+,
\]
and
\begin{equation}\label{eq:Xminus}
X^-(r,\theta)=
\left(
-\frac{2\rho r}{\rho^2+r^2}\cos\theta,
-\frac{2\rho r}{\rho^2+r^2}\sin\theta,
\frac{\rho^2-r^2}{\rho^2+r^2}
\right).
\end{equation}
In particular,
\[
|u_r^\pm(1,\theta)|=R,
\qquad
u_r^\pm(1,\theta)\cdot u_\theta^\pm(1,\theta)=0,
\]
and, for $0<R\le1$, the boundary curves of $X^\pm$ are horizontal circles of radius $R$ at heights $\pm\sqrt{1-R^2}$, with the two heights
coinciding in the equatorial case $R=1$.
\end{proposition}

\begin{proof}
We first justify the stereographic formulas. Using the south-pole stereographic projection, substituting $w=\rho z=\rho r e^{i\theta}$ into
$\pi_S^{-1}$ gives the displayed formula for $u^+$. Substituting $w=\rho/\bar z=(\rho/r)e^{i\theta}$ gives the displayed formula for $u^-$ for
$r>0$. At $r=1$, the identities
\[
 R=\frac{2\rho}{1+\rho^2},
 \qquad
 \sqrt{1-R^2}=\frac{1-\rho^2}{1+\rho^2}
\]
give $u^\pm(1,\theta)=g_R(e^{i\theta})$, so both stereographic formulas have precisely the required trace. Although $\rho/\bar z$ has a pole
at the origin, this pole represents the south pole under the inverse stereographic coordinate; equivalently,
\[
 u^-(x,y)=\left(\frac{2\rho x}{\rho^2+x^2+y^2},\frac{2\rho y}{\rho^2+x^2+y^2},\frac{x^2+y^2-\rho^2}{\rho^2+x^2+y^2}\right),
\]
extends smoothly across $0$.

The harmonic-map equations and the potential identities now follow by direct differentiation. For $u^+$ one checks that \eqref{eq:X-def} holds
with $X^+=u^+$. For $u^-$ a direct computation shows that \eqref{eq:X-def} holds with $X^-$ given by \eqref{eq:Xminus}. Evaluating the
derivatives at $r=1$ and using the identities relating $R$ and $\rho$ yields the stated boundary identities.
\end{proof}

\section{Two vector identities}\label{app:vectors}

\begin{lemma}\label{lem:vector-identities}
For $a,b,c\in\R^3$ one has
\[
(a\times b)\times(a\times c)=\det(a,b,c)\,a.
\]
If $v,w\in T_u\Sph$ for some $u\in\Sph$, then
\[
v\times w=\bigl(u\cdot(v\times w)\bigr)u.
\]
\end{lemma}

\begin{proof}
For the first identity, apply the triple-product formula
\[
x\times(y\times z)=y(x\cdot z)-z(x\cdot y)
\]
with $x=a\times b$, $y=a$, and $z=c$. This gives
\[
(a\times b)\times(a\times c)
=a\bigl((a\times b)\cdot c\bigr)-c\bigl((a\times b)\cdot a\bigr)
=\det(a,b,c)\,a.
\]
For the second identity, $v$ and $w$ are tangent to the sphere at $u$, hence orthogonal to $u$. Their cross product is therefore orthogonal to
the tangent plane and must be parallel to the unit normal $u$.
\end{proof}

\makeatletter
\begingroup
\def\@tocwrite#1#2{}%

\noindent \textbf{Declarations of interest}: None.

\bigskip
\noindent \textbf{Data availability statement}: There are no new data associated with this article.

\bigskip
\noindent \textbf{AI assistance statement}: The authors used OpenAI models to assist with initial conceptualization, symbolic checking and
manuscript editing; all mathematical validation, final proof decisions, and final wording remain the sole responsibility of the human authors.

\end{document}